\documentclass[12pt]{amsart}

\usepackage{amstext}
\usepackage{amsthm}
\usepackage{amsmath}
\usepackage{amssymb}
\usepackage{latexsym}
\usepackage{amsfonts}
\usepackage{graphicx}
\usepackage{texdraw}
\usepackage{graphpap}

\usepackage[pagebackref,hypertexnames=false, colorlinks, citecolor=red, linkcolor=red]{hyperref} 
\usepackage[backrefs]{amsrefs}

\input txdtools

\begin{document}


\title[Thin Interpolating Sequences]{Thin Sequences}


\author[P. Gorkin]{Pamela Gorkin$^\dagger$}
\address{Pamela Gorkin, Department of Mathematics\\ Bucknell University\\  Lewisburg, PA  USA 17837}
\email{pgorkin@bucknell.edu}
\thanks{$\dagger$ Research supported in part by Simons Foundation Grant 243653}

\author[B.D.Wick]{Brett D. Wick$^\ddagger$}
\address{Brett D. Wick, Department of Mathematics\\ Washington University - St. Louis\\ One Brookings Drive\\ St. Louis, MO USA 63130-4899}
\email{wick@math.wustl.edu}
\thanks{$\ddagger$ Research supported in part by National Science Foundation DMS grants \# 0955432 and \#1500509.}


\subjclass[2010]{Primary: 46E22.  Secondary:  30D55, 47A, 46B15.}


\keywords{reproducing kernel, thin sequences, interpolation, asymptotic orthonormal sequence}

%
%

\newcommand{\ci}[1]{_{ {}_{\scriptstyle #1}}}

\newcommand{\norm}[1]{\ensuremath{\left\|#1\right\|}}
\newcommand{\abs}[1]{\ensuremath{\left\vert#1\right\vert}}
\newcommand{\p}{\ensuremath{\partial}}
\newcommand{\pr}{\mathcal{P}}

\newcommand{\pbar}{\ensuremath{\bar{\partial}}}
\newcommand{\db}{\overline\partial}
\newcommand{\D}{\mathbb{D}}
\newcommand{\B}{\mathbb{B}}
\newcommand{\Sp}{\mathbb{S}}
\newcommand{\T}{\mathbb{T}}
\newcommand{\R}{\mathbb{R}}
\newcommand{\Z}{\mathbb{Z}}
\newcommand{\C}{\mathbb{C}}
\newcommand{\N}{\mathbb{N}}
\newcommand{\Hc}{\mathcal{H}}
\newcommand{\scrL}{\mathcal{L}}
\newcommand{\td}{\widetilde\Delta}

\newcommand{\CC}{\mathcal{C}}
\newcommand{\EE}{\mathcal{E}}
\newcommand{\RR}{\mathcal{R}}
\newcommand{\TT}{\mathcal{T}}

\newcommand{\La}{\langle }
\newcommand{\Ra}{\rangle }
\newcommand{\rk}{\operatorname{rk}}
\newcommand{\card}{\operatorname{card}}
\newcommand{\ran}{\operatorname{Ran}}
\newcommand{\osc}{\operatorname{OSC}}
\newcommand{\im}{\operatorname{Im}}
\newcommand{\re}{\operatorname{Re}}
\newcommand{\tr}{\operatorname{tr}}
\newcommand{\vf}{\varphi}

\renewcommand{\qedsymbol}{$\Box$}
\newtheorem{thm}{Theorem}[section]
\newtheorem{lm}[thm]{Lemma}
\newtheorem{cor}[thm]{Corollary}
\newtheorem{conj}[thm]{Conjecture}
\newtheorem{prob}[thm]{Problem}
\newtheorem{prop}[thm]{Proposition}
\newtheorem*{prop*}{Proposition}
\newtheorem{defin}[thm]{Definition}
\theoremstyle{remark}
\newtheorem{rem}[thm]{Remark}
\newtheorem*{rem*}{Remark}
\newtheorem{exam}{Example}
\newtheorem{question}{Question}

\numberwithin{equation}{section}


%
%

\begin{abstract}
We look at thin interpolating sequences and the role they play in uniform algebras, Hardy spaces, and model spaces.\end{abstract}

%
%

\maketitle

``One of the striking successes of the function algebra viewpoint has been in the study of the algebra $H^\infty$ of all bounded holomorphic functions on the unit disk $\mathbb{D}$. Not only are the results which have been obtained deep but the questions raised have also enriched the classical study of the boundary behavior of holomorphic functions.'' R. G. Douglas, Mathematical Reviews, (MR0428044 and MR428045)

\section{The beginning of interpolation in Hardy spaces}

R. C. Buck proposed the idea of characterizing interpolating sequences for $H^\infty$, the algebra of bounded analytic functions on the open unit disk $\mathbb{D}$, via an explicit condition on the sequence, \cite{SS}. Buck conjectured that if a sequence of points in $\mathbb{D}$ approached the boundary quickly enough, it would be interpolating for the algebra $H^\infty$; that is, for all $\{w_n\} \in \ell^\infty$ there exists $f \in H^\infty$ such that $f(z_n) = w_n$ for all $n$. We discuss briefly the background on interpolating sequences before turning to thin interpolating sequences.

In 1958, W. Hayman \cite{Hayman} proved the following theorem.

\begin{thm}[Hayman]\label{Hayman} A necessary condition for a sequence $\{z_n\}$ to be an interpolating sequence is that there exist a constant $C > 0$ such that 
\begin{eqnarray}\label{thm:C}
\inf_{n} \prod_{m: m \ne n} \left|\frac{z_m - z_n}{1 - \overline{z_n} z_m}\right| \ge C.
\end{eqnarray}

A sufficient condition is that there exist $\lambda < 1$ and $C_1 > 0$ so that
$$\inf_n \prod_{m: m \ne n} \left[1 - \left(1 - \left|\frac{z_m - z_n}{1 - \overline{z_n} z_m}\right|^\lambda\right)\right] \ge C_1.$$ \end{thm}

Hayman wrote, ``It seems quite possible that \eqref{thm:C} is in fact sufficient as well as necessary, but I have been unable to prove this.'' Hayman's proof was constructive and provided other very useful estimates on the sequence. Independently and also in 1958, Carleson \cite{C} presented the  condition Buck anticipated: If $\{z_n\}$ is a Blaschke sequence and $B$ the corresponding Blaschke product, then $\{z_n\}$ is interpolating if there exists $\delta > 0$ such that 
 
\begin{eqnarray}\label{Carleson}\inf_n \prod_{m \ne n}\left|\frac{z_m - z_n}{1 - \overline{z_m} z_n}\right| = \inf_n (1 - |z_n|^2)|B^\prime(z_n)| \ge \delta > 0.\end{eqnarray}

In 1961, Shapiro and Shields \cite{SS} considered interpolation in the Hardy space $H^p$ and described it as a weighted interpolation problem. For $f \in H^p$, they defined an operator $T_p$ by
\[T_pf = \left\{f(z_j)(1 - |z_j|^p)^{1/p}\right\}_{j = 1}^\infty,\]

\smallskip
\noindent
and asked when $T_p(H^p) = \ell_p$, where $\ell_p$ is the space of $p$-summable sequences; when $p = \infty$, this is precisely the requirement that $\{z_n\}$ be interpolating for $H^\infty$. In the theorem below, $\delta_{jk} = 1$ if $j = k$ and $0$ otherwise.
 
 \begin{thm}[Shapiro, Shields, 1961]\cite{SS} For $1 \le p \le \infty$ the necessary and sufficient condition for $T_p H^p = \ell^p$ is that there exist functions $f_k \in H^p$ such that 
 \begin{enumerate}
 \item $f_k(z_j)(1 - |z_j|^2)^{1/p} = \delta_{jk}$;
 \item $\|f_k\|_p \le 1/\delta$.
 \end{enumerate}
 \end{thm}
 
 They note that in $H^1$ an interpolating function can be given explicitly.   If we assume that $\displaystyle\sum_{k=1}^{\infty} |w_k|(1 - |z_k|^2) < \infty$, then letting $B$ denote the Blaschke product corresponding to $\{z_n\}$,
 
\[f(z) = \sum_{k} \frac{B(z)}{B^\prime(z_k)}\left(\frac{1}{z-z_k} + \frac{\overline{z_k}}{1 - \overline{z_k}z}\right)\] solves the problem.

In addition, it was also known that the interpolating function $f$ could be chosen to satisfy 
\[|f(z)| \le \frac{2}{\delta^5}\left(1 + 2 \log\frac{1}{\delta}\right)\sup_n |w_n|.\]

Three natural questions then arise: First, when interpolating in $H^p$ what is the best bound on the $p$-norm of an interpolating function? We see that, as $\delta \to 0$, the corresponding upper bounds approach infinity, which is expected, but as $\delta \to 1$, this point approaches twice what one would hope it would approach. Second, in view of Shapiro and Shield's explicit demonstration of a function in $H^1$ that does the interpolation, another natural question is whether we can exhibit the function explicitly. Finally, we might ask if these two questions can be combined; that is, can we explicitly exhibit the function of best norm that does the interpolation?

\section{Best bounds and best functions}

If we have an interpolating sequence for $H^\infty$, then the map $T: H^\infty \to \ell^\infty$ defined by $T(f) = \{f(z_n)\}$ is a bounded surjective map. If we let $B$ denote the corresponding Blaschke product, then $T$ induces a bijective map $T: H^\infty/BH^\infty \to \ell^\infty$.  Thus, as a consequence of the open mapping theorem, given $w:= \{w_n\} \in \ell^\infty$, there is a constant $C$ such that
\[\|f\|_\infty \le C\|w\|_\infty.\] The smallest such constant is called {\it the constant of interpolation} and depends on the separation constant $\delta$, so it is often denoted by $M(\delta)$. So we are after three things: An estimate on $M(\delta)$, an explicit expression for a function that does the interpolation, and the connection between the two. We will begin with Earl's estimates.

In 1970, J. P. Earl \cite{E} showed that if an interpolating sequence $\{z_n\}$ satisfies \eqref{Carleson} then for any $M$ such that
\begin{eqnarray}\label{Earl}
M > \frac{2 - \delta^2 + 2(1 - \delta^2)^{1/2}}{\delta^2} \sup_{n} |w_n|\end{eqnarray}
there exists a Blaschke product $B$ such that 
\[M e^{i \alpha} B(z_j) = w_j~\mbox{for all}~j.\]

Though Earl's result shows roughly where the zeros of the Blaschke product lie, it does not give an explicit expression for the function that does the interpolation. Now, if we restrict our sequences $\{w_j\}$ so that $\|w\|_\infty \le 1$, let $\delta = \delta(B) = \inf_n (1 - |z_n|^2)|B^\prime(z_n)|$ and use a normal families argument, we obtain a function of norm at most 
\[\frac{2 - \delta^2 + 2(1 - \delta^2)^{1/2}}{\delta^2} = \left(\frac{1 + (1 - \delta^2)^{1/2}}{\delta}\right)^2\] 
that does the interpolation. In particular, Earl's theorem gives us an estimate on $M(\delta)$.

If the original sequence $\{z_j\}$ has the property that the corresponding Blaschke product $C$ satisfies $\delta_j(C) = |C_j(z_j)| \to 1$, then the function that does the interpolation may be chosen to have this property as well; that is, it can be chosen to be a so-called {\it thin} Blaschke product (this result was stated in \cite{DN} with a general plan of attack; details appear in \cite{Mo}). The proof involves adapting the proof of J. P. Earl to this situation. In general, however, the ``best'' function that does the interpolation may not be a unimodular constant times a Blaschke product, but we see that the closer $\delta$ is to $1$, the closer the norm of $f$ is to $1$.  Thus, interpolating Blaschke products for which $\delta_j(B) \to 1$ as $j \to \infty$ would seem to have particularly interesting properties and, indeed, they have been closely studied. Such sequences are called {\it thin sequences} and, if we require that they also be interpolating, they are {\it thin interpolating sequences}. 

\begin{defin} A Blaschke product $B$ with zero sequence $\{z_n\}$ satisfying 
\[\lim_n |B_n(z_n)| = \lim_n  \prod_{m \ne n}\left|\frac{z_m - z_n}{1 - \overline{z_m} z_n}\right| = \lim_n (1 - |z_n|^2)|B^\prime(z_n)| = 1\]
is called a thin Blaschke product and the sequence $\{z_n\}$ is said to be a thin sequence.
\end{defin}

Note that thin sequences may have finitely many points that appear finitely many times, but they cannot repeat infinitely many points. We will assume, unless otherwise stated,  that our thin sequences are interpolating sequences, so that points are distinct. We now present some examples of thin sequences. Recall that for $z, w \in \mathbb{D}$ the pseudohyperbolic distance between $z$ and $w$ in $\mathbb{D}$ is $\rho(z,w) = \left|\frac{z - w}{1 - \overline{w}z}\right|$.

\begin{exam} Every Blaschke sequence has a thin subsequence. \end{exam}

\begin{proof} Let $\varepsilon_j$ be an increasing sequence with $0 < \varepsilon_j < 1$ so that $\delta_j = \prod_{k \ne j} \varepsilon_k \to 1$. Beginning with $z_{n_1} = z_1$, choose $z_{n_2}$ so that $\rho(z_1, z_{n_2}) \ge  \varepsilon_2$. Assuming $z_{n_1}, \ldots, z_{n_j}$ have been chosen, we choose $z_{n_{j + 1}}$ with $\rho(z_{n_k}, z_{n_{j+1}}) > \varepsilon_{j + 1}$ for all $k \le j$. If $C$ is the Blaschke product corresponding to $\{z_{n_j}\}$, then 
$$|C_j(z_{n_j})| \ge \prod_{k \ne j} \rho(z_{n_k}, z_{n_j}) \ge \prod_{k \ne j} \varepsilon_k \to 1,$$ as $j \to \infty$. \end{proof}

The next example below is due to W. Hayman who, in working towards establishing a condition that a sequence be interpolating, proved the following theorem:

\begin{thm}[W. Hayman]\label{Hayman2} A sufficient condition for a sequence of distinct points, $\{z_n\}$, to be interpolating is that 
$$\limsup_{n\to \infty} \frac{1 - |z_{n+1}|}{1 - |z_n|} < 1.$$ If $0 < z_n < 1$ and $\{z_n\}$ is increasing, then the condition is also necessary.\end{thm}

This theorem follows from the previous theorem due to Hayman, Theorem~\ref{Hayman}, above. A careful inspection of his proof establishes the following.

\begin{exam} Let $\{z_n\}$ be a sequence of distinct points satisfying 
$$\frac{1 - |z_{n+1}|}{1 - |z_n|} \to 0.$$ Then $\{z_n\}$ is a thin sequence.
\end{exam}

Here is a rough idea of why this is true (see \cite[Proposition 4.3 (i)]{CFT} for a different proof): Let $k \in \mathbb{N}$ and suppose 
$$ 1 - |z_{n + 1}| \le c_k (1 - |z_n|)~\mbox{for}~ n \ge k.$$ Note that we are assuming that $c_k \to 0$. 
Now for points $z, w$ in $\mathbb{D}$ the pseudohyperbolic distance between $z$ and $w$ satisfies
$$\rho(z, w) \ge \frac{|z|-|w|}{1 - |\bar{z} w|}.$$  Therefore, if we break the product $\prod_{j \ne k} \left|\frac{z_k - z_j}{1 - \overline{z_j} z_k}\right|$ into two pieces, for $j > k$ we will have
$$1 - |z_j| \le c_k^{j - k}(1 - |z_k|).$$ Consequently $$|z_j| - |z_k| \ge (1 - c_k^{j - k})(1 - |z_k|).$$ But
$$1 - |\overline{z_j} z_k| \le (1 + c_k^{j - k})(1 - |z_k|).$$ Thus, 
$$\prod_{j > k} \left|\frac{z_k - z_j}{1 - \overline{z_j} z_k}\right| \ge \frac{1 - c_k^{j - k}}{1 + c_k^{j-k}}.$$ 

For $j < k$, we have $$1 - |z_{j + 1}| \le c_j (1 - |z_j|).$$ Thus, $$1 - |z_k| \le \prod_{l = j}^{k - 1} c_l(1 - |z_j|)~\mbox{and}~ |z_k| - |z_j| \ge \left(1 - \prod_{l = j}^{k - 1} c_l\right)(1 - |z_j|),$$ while
$$1 - |\overline{z_k} z_j| \le \left(1 + \prod_{l = j}^{k - 1} c_l\right)(1 - |z_j|).$$ So
$$\prod_{j < k} \left|\frac{z_k - z_j}{1 - \overline{z_j} z_k}\right| \ge \frac{1 - \prod_{l = j}^{k - 1} c_l}{1 + \prod_{l = j}^{k - 1} c_l}.$$ 
 From this, we conclude that the sequence $\{z_n\}$ is thin.\\

Our last example is reminiscent of a result of Naftalevitch that says that any Blaschke sequence can be rotated to form an interpolating sequence. Naftalevitch's theorem is also a consequence of the following theorem, which can be found in \cite[Proposition 4.3 (ii)]{CFT}.

\begin{exam} Let $\{z_n\}$ be a Blaschke sequence. Then there is a thin interpolating sequence $\{w_n\}$ with $|w_n| = |z_n|$.\end{exam}

The proof actually constructs the sequence.  Supposing that $\{r_n\}$ is increasing, choose a sequence of positive numbers $b_n$ with $\frac{1 - r_n}{b_{n+1}} \to 0$. Without loss of generality, we may assume that $\displaystyle\sum_{n=1}^{\infty} b_n < \frac{\pi}{2}$. Let $\theta_n = \displaystyle\sum_{k = 1}^n b_k$ and $\lambda_n = r_n e^{i \theta_n}$. Then, as shown in \cite{CFT}, this sequence is thin.

This example is quite similar to an example developed  by Joel H. Shapiro in the context of composition operators. So recall that if $T$ is an analytic self-map of the unit disk, a composition operator on the Hardy space $H^2$ is defined by $C_T(f) = f \circ T$. Once one has checked that the operator is bounded, which it is (see, for example, \cite{S}), it is natural to study when it is compact. It is now known \cite{S} that if $C_T$ is compact, then the angular derivative of $T$ exists at no point of the unit circle. Shapiro's example was developed to show that the angular derivative condition is not sufficient to imply that the operator is compact. That the Blaschke product Shapiro constructed is actually thin was first noticed by D. Suarez and proved in \cite{GG}, but the application of ideas in \cite{CFT} simplifies the proof significantly.

\begin{exam}\cite[p. 184]{Shapiro} Wrap intervals $I_k$ of length $\frac{1}{k}$ about the unit circle, placing a zero of the Blaschke product $B$ at the point $\left(1 - \frac{1}{k^2}\right)e^{i\theta_k}$, where $e^{i \theta_k}$ is the center of the arc $I_k$. Then $B$ is a thin Blaschke product for which the angular derivative does not exist at any point of the unit circle.\end{exam}

Though out of sequence chronologically, we complete this section with answers to the questions we discussed above. The first is an explicit description of the functions that solve the interpolation problem.

P. Beurling \cite{C1} showed that given an interpolating sequence $\{z_j\}$, there exist functions $f_j$ in $H^\infty$ with the property that $f_j(z_j) = 1$ and $f_j(z_k) = 0$ for $j \ne k$ such that 
\begin{eqnarray}\label{beurling} \sup_{z \in \mathbb{D}} \sum_j |f_j(z)| < \infty. \end{eqnarray} In fact, the bound in \eqref{beurling} can be connected to the separation constant $\delta$, but the functions are not given explicitly. Instead, the first explicit description is due to Peter Jones and appeared in 1983, \cite{J}. The functions are given as follows:

\[
f_j(z) := \frac{B_j(z)}{B_j(z_j)}\left(\frac{1-\abs{z_j}^2}{1-\overline{z_j}z}\right)^2 e^{\left(-\frac{1}{2C(\delta)}\sum_{\abs{z_m}\geq\abs{z_j}}\left(\frac{1+\overline{z_m}z}{1-\overline{z_m}z}-\frac{1+\overline{z_m}z_j}{1-\overline{z_m}z_j}\right)(1-\abs{z_m}^2)\right).}
\]

It is easy to see that $f_j(z_j) = 1$ and $f_j(z_k) = 0$ for $j \ne k$. It is, of course, much harder to see that 
\[\sum_{j=1}^{\infty} |f_j(z)| < M~\mbox{for all}~z \in \mathbb{D},\] but it is true and can be shown using a computation that culminates in a Riemann sum that yields the result.  

Thus for any $a\in\ell^\infty$, if we let $\displaystyle 
f(z)=\sum_{j=1}^\infty a_j f_j(z)\in H^\infty(\mathbb{D})
$ we see that
\[f(z_j)=a_j; \quad \abs{f(z)}\leq \|a\|_{\ell^\infty}\left(\sum_{j=1}^\infty \abs{f_j(z)}\right)\leq C(\delta)\|a\|_{\ell^\infty}.
\]

In 2004,  Nicolau, Ortega-Cerd\`{a},  and Seip \cite{NOS}, modifying the explicit formulas given by Jones, were able to provide sharp upper and lower bounds on the interpolation constant.   For thin sequences, the following version of P. Beurling's theorem was proven in \cite{GPW} using the commutant lifting theorem.

\begin{thm}\label{thm:replacement}
Let $\{z_n\}$ be a thin sequence. Then for every $\varepsilon > 0$ there exists $N$  such that for $n \ge N$ there exist $f_n \in H^\infty$ such that for $j, k \ge N$ we have $$f_n(z_n) = 1~\mbox{and}~ f_n(z_k) = 0,~\mbox{for}~ j \ne k,$$ and $$\sup_{z \in \mathbb{D}} \sum_{n \ge N} |f_n(z)| < (1 + \varepsilon).$$ In particular, for every sequence $a \in \ell^\infty$ with $\|a\|_{\ell^\infty} \le 1$ the function $g_a$ defined by $g_a(z):= \sum_{n \ge N} a_n f_n(z) \in H^\infty$ satisfies $\|g_a\|_\infty \le  (1 + \varepsilon)\|a\|_{N, \ell^\infty}$ and $g_a(z_j) = a_j$ for $j \ge N$.
\end{thm}

This theorem can be proved using Earl's estimate, but the proof requires a brief introduction to the maximal ideal space of $H^\infty$. We turn to that introduction now and return to Theorem~\ref{thm:replacement} once we have established the basics.

\section{The maximal ideal space}\label{Maximal Ideal Space}

One of the tools that is most useful in the study of $H^\infty$ as a uniform algebra is  its {\it maximal ideal space}, $M(H^\infty)$, or the space of nonzero multiplicative linear functionals. It is called the maximal ideal space because the kernel of a nonzero multiplicative linear functional is a maximal ideal and, conversely, every maximal ideal is the kernel of a multiplicative linear functions.

When endowed with the weak-$\ast$ topology, $M(H^\infty)$ is a compact Hausdorff space.  By studying certain partitions of the maximal ideal space, mathematicians were able to shed light on the behavior of functions in $H^\infty$.  The disk, $\mathbb{D}$, can be identified with a subset of the maximal ideal space, by identifying the point $z$ with the functional that is point evaluation at $z$ and Carleson's corona theorem says that the disk is dense in $M(H^\infty)$. 

That, plus the following theorem, are the last two ingredients that we need in our proof of Theorem~\ref{thm:replacement}.

\begin{thm}\label{intestimate} Let $A$ be a uniform algebra on a compact space $X$ and let $\{x_1, \ldots, x_n\}$ be a finite set of points in $X$. If
$$M = \sup_{\|a\|_\infty \le 1} \inf\left\{\|g\|_{A}: g \in A, g(x_j) = a_j, j = 1, 2, \ldots, n\right\},$$ then for every $\varepsilon > 0$ there are functions $f_j \in A$ for which 
$$f_j(x_j) = 1 ~\mbox{and}~ f_j(x_k) = 0 \quad \mbox{for}~k\ne j$$ and 
$$\sup_{x \in X} \sum_{j = 1}^n |f_j(x)| \le M^2 + \varepsilon.$$
\end{thm}

In our case, the compact space $X$ is $M(H^\infty)$ and our points $x_j$ will be $z_j \in \mathbb{D}$.  Note that in this case -- that is, when $A = H^\infty$ -- a normal families argument implies that we can find a sequence $f_j$ such that $f_j(z_j) = 1$, $f_j(z_k) = 0$ and $\sum_{z \in \mathbb{D}} |f_j(z)| \le M^2$.

\begin{proof}[Proof of Theorem~\ref{thm:replacement}] 
Recall that $\delta_n = |B_n(z_n)|$. Let $\varepsilon > 0$ be given. Then, since $\delta_n \to 1$, there exists $N$ such that 
\begin{equation}\label{delta}
\left(\frac{2 - \delta_n^2 + 2(1 - \delta_n^2)^{1/2}}{\delta_n^2}\right)^2 < 1 + \varepsilon ~\mbox{for}~ n \ge N.
\end{equation}
 Consider the sequence $\{z_j\}_{j \ge N}$ and let $\delta$ be the separation constant for this sequence. By Earl's estimate, \eqref{Earl}, we know that given $w \in \ell^\infty$ with $\|w\|_\infty \le 1$, there exists a function $f \in H^\infty$ such that $f(z_j) = w_j$ for all $j \ge N$ and $\|f\|_{\infty} \le \frac{2 - \delta^2 + 2(1 - \delta^2)^{1/2}}{\delta^2}$. By \eqref{delta} we see that $\|f\|_\infty \le 1 + \varepsilon$. Therefore, by Theorem~\ref{intestimate} and a normal families argument, we know that there are functions $f_{j}$ such that for $j, k \ge N$ we have
$$f_{j}(z_j) = 1, f_{j}(z_k) = 0~\mbox{for $j \ne k$ and}~ \sup_{z \in \mathbb{D}} \sum_{j = N}^\infty |f_{j}(z)| \le  1 + \varepsilon.$$ 
\end{proof}

The maximal ideal space is particularly useful here, but in spite of our familiarity with the open dense set $\mathbb{D}$, the set $M(H^\infty) \setminus \mathbb{D}$ is difficult to understand. As luck would have it, and as Sarason showed, the space $H^\infty + C$, consisting of sums of (boundary) functions in $H^\infty$ and continuous functions on the unit circle $\mathbb{T}$ is a closed subalgebra \cite{S} of $L^\infty$ and the maximal ideal space is $M(H^\infty + C) = M(H^\infty) \setminus \mathbb{D}$ -- precisely the set we don't understand well.

It turns out that the analytic structure that we have in $\mathbb{D}$ does, in a certain sense to be made precise, carry over to $M(H^\infty)$ as we see from the {\it Gleason parts}.   For $\varphi_1, \varphi_2 \in M$ {\it the pseudohyperbolic distance} is
\[\rho(\varphi_1, \varphi_2) = 
\sup\{|\hat{f}(\varphi_2)| : f \in H^\infty, {\|f\|}_\infty \le 1, \hat{f}(\varphi_1) = 0\}.\] Points are in the same Gleason part if $\rho(\varphi_1, \varphi_2) < 1$ and this defines an equivalence relation on $M(H^\infty)$; the equivalence classes are the Gleason parts. One equivalence class is the unit disk and the others lie in $M(H^\infty+C)$. In trying to understand the parts in $M(H^\infty) \setminus \mathbb{D}$, Hoffman considered, for each point $\alpha \in \mathbb{D}$, the map $L_\alpha(z) = (z + \alpha)(1 + \overline{\alpha}z)^{-1}$. Given a net of points $(\alpha_\beta)$ converging to a point $\varphi$ in $M(H^\infty)$, the corresponding maps $L_{\alpha_\beta}$ converge to a one-to-one map $L_\varphi: \mathbb{D} \to M(H^\infty)$ and $L_\varphi(\mathbb{D}) = P(\varphi)$, the Gleason part of $\varphi$. The map $L_\varphi$ imparts an analytic structure on $P(\varphi)$: If $f \in H^\infty$, we can define the Gelfand transform of $f$, denoted $\hat{f}$, on $M(H^\infty)$ by $\hat{f}(\varphi) = \varphi(f)$ and then $\hat{f} \circ L_\varphi$ is an analytic function on $\mathbb{D}$. It is customary to drop the ``hat'' and refer to $f$ even when using the function $\hat{f}$. 

One of Hoffman's goals was to show that a point $\varphi \in M(H^\infty + C)$ is in the closure of an interpolating sequence if and only if the map $L_\varphi$ is not constant. Though Hoffman's work allowed mathematicians to use the analytic structure of the parts as a tool, that does not mean that the parts are tractable. For example, parts may or may not look like the disk, but the points $\psi$ that lie in the closure of a thin part are always homeomorphic to the disk, as noted by Hoffman in his seminal paper, \cite{Ho}.

\begin{prop}\label{Hoffman} Let $\{\alpha_n\}$ denote a thin sequence and $B$ the corresponding Blaschke product. Then for any point $\varphi \in M(H^\infty + C)$ in the closure of $\{\alpha_n\}$, the Gleason part is homeomorphic to the unit disk.
\end{prop}

\begin{proof} By assumption, we know that $\displaystyle\lim_n(1 - |\alpha_n|^2)|B^\prime(\alpha_n)| = 1$. Then $(\hat{B} \circ L_{\alpha_n})(0) = 0$ for all $n$ and $|(\hat{B} \circ L_{\alpha_n})^\prime(0)| = |B^\prime(\alpha_n)|(1 - |\alpha_n|^2) \to 1$. So whatever $\hat{B} \circ L_\varphi$ is, we know that $(\hat{B} \circ L_\varphi): \mathbb{D} \to \mathbb{D}$, $\hat{B} \circ L_\varphi(0) = 0$, and $|(\hat{B} \circ L_\varphi)^\prime(0)| = 1$. Since $\hat{B} \circ L_\varphi$ is also analytic, Schwarz's lemma shows that $\hat{B} \circ L_\varphi(z) = \lambda z$ for some $\lambda$ of modulus $1$. Therefore if $\varphi$ is in the closure of a thin sequence, $L_\varphi$ is a homeomorphism and its inverse is a (unimodular) constant multiple of $\hat{B}$. \end{proof}

To the best of our knowledge, this is the first appearance of thin sequences in the literature, and this has interesting implications.  By Hoffman's work, it turns out that a Blaschke product for which the zeros form a thin interpolating sequence are  {\it indestructible}; that is, if you take an automorphism $T_a(z) = \lambda \frac{z - a}{1 - \overline{a} z}$ (with $a \in \mathbb{D}$ and $\lambda$ in the unit circle, $\mathbb{T}$) and consider $T_a \circ B$, this will again be a thin Blaschke product (though finitely many zeros may be repeated).  This fact and Proposition~\ref{Hoffman} imply that if a thin Blaschke product is of modulus less than one on a part, then it has exactly one zero on that part.

\begin{prop} If $B = B_1 B_2$ is a factorization of a thin Blaschke product, then for each $\varphi \in M(H^\infty + C)$ either $|B_1\circ L_\varphi(z)| = 1$ for all $z \in \mathbb{D}$ or $|B_2 \circ L_\varphi(z)| = 1$ for all $z \in \mathbb{D}$. \end{prop}

\begin{proof} Let $\varphi \in M(H^\infty + C)$. If $|B_1(\varphi)| < 1$, then $C_1:=T_{B_1(\varphi)}\circ B_1$ is still thin and $\hat{C_1} \circ L_\varphi(z) = \lambda_1 z$. Similarly, if $|B_2(\psi)| < 1$ for some $\psi \in P(\varphi)$, then for the corresponding Blaschke product $C_2$, we have that $\hat{C_2} \circ L_\psi(z) = \lambda_2 z$. In particular, $B$ would have two zeros (counted according to multiplicity) in $P(\varphi)$ and that is impossible.
\end{proof}

Thus, thin sequences have zeros that are pseudo-hyperbolically far apart in the disk as well as in $M(H^\infty + C)$ and it is this separation that made them particularly interesting sequences for the study of interpolation. 

\section{Thin sequences, interpolation and uniform algebras}

The strength of the separation of points in the closure of thin sequences is illustrated by a result of T. Wolff. To place it in its proper context, we need to understand what happened in the study of closed subalgebras of $L^\infty$ containing $H^\infty$; the so-called {\it Douglas algebras}, in honor of R. G. Douglas who conjectured that every such algebra is generated by $H^\infty$ and the complex conjugates of inner functions invertible in the algebra. That this is true for $H^\infty + C$ is Sarason's theorem: the only invertible inner functions in $H^\infty + C$ are finite Blaschke products and Sarason showed that $H^\infty + C = H^\infty[\overline{z}]$. After Douglas made his conjecture, he and Rudin \cite{DR} showed that $L^\infty$ is of the right form. The final result was even better than what Douglas conjectured: Chang and Marshall \cite{Ch, Ma} showed that every such algebra was generated by $H^\infty$ and the conjugates of the interpolating Blaschke products invertible in that algebra.

The proof is divided into two pieces. Chang showed that if two Douglas algebras had the same maximal ideal space, then they were the same Douglas algebra,while Marshall showed that if $A$ is a Douglas algebra and $A_I$ is the (closed) algebra generated by $H^\infty$ and the complex conjugates of the interpolating Blaschke products invertible in $A$, then the maximal ideal space of $A$, denoted $M(A)$, is equal to the maximal ideal space of $A_I$. Their work requires an understanding of how elements $\varphi \in M(H^\infty)$ ``work.''  Each $\varphi \in M(H^\infty)$ can be defined by integration against a positive measure with closed support in the maximal ideal space of $L^\infty$; that is,
$$\varphi(f) = \int_{{\mbox{\tiny supp}} \,{\varphi}} f d\mu_\varphi,$$ and given a Douglas algebra, we may think of $M(A)$ as a subset of $M(H^\infty)$; $M(A)$ can be identified with the multiplicative linear functionals in $M(H^\infty)$ for which the representing measures are multiplicative on $A$. (See, for example, \cite[Chapter IX]{Garnett}.)

Sticking with the uniform algebra point of view for a moment, one might wonder what happens when one looks at the closed algebra $A$ of $H^\infty$ and the conjugates of all thin interpolating Blaschke products. Hedenmalm \cite{H} showed that an inner function is invertible in $A$ if and only if it is a finite product of thin interpolating Blaschke products.

As this suggests, thin interpolating sequences are very well behaved. Wolff and Sundberg \cite{W}, \cite{SW} showed, among other things, that these sequences are the interpolating sequence for the (very small) algebra $QA = \overline{H^\infty + C} \cap H^\infty$ (here the bar denotes the complex conjugate).  This algebra acts, in many ways, like the disk algebra (for this, \cite{W} is a good resource). We start with the algebra of quasi-continuous functions: let $QC = (H^\infty + C) \cap \overline{H^\infty + C}$. The algebra $QA$ is then $QA: = QC \cap H^\infty$. The theorem we concentrate on here is the following:

\begin{thm}[Wolff, Wolff-Sundberg] \label{SW}The following are equivalent for an interpolating sequence $\{z_n\}$.
\begin{enumerate}
\item For any $\{\lambda_n\} \in \ell^\infty$ there is $f \in QA$ with $f(z_n) = \lambda_n$;
\item For any $\{\lambda_n\} \in \ell^\infty$, $\varepsilon > 0$, then there is an $f \in H^\infty$ with $\|f\|_{\infty} < \displaystyle\limsup_{n\to\infty} |\lambda_n| + \varepsilon$ and $f(z_n) = \lambda_n$ all but finitely many $n$;
\item $\displaystyle\lim_{n \to \infty} \prod_{m \ne n} \left| \frac{z_n - z_m}{1 - \overline{z_m} z_n}\right| = 1$.
\end{enumerate} 
\end{thm}

Thus, thin sequences are interpolating sequences for a very small algebra and, therefore, they must have a strong separation property. One way to think of this separation property is in the maximal ideal space. We first describe the most natural partition of $M(H^\infty + C)$, namely the fibers. 

\begin{defin}
For $\lambda \in \mathbb{T}$, let $M_\lambda = \{\varphi \in M(H^\infty + C): \varphi(z) = \lambda\}$. The set $M_\lambda$ is called the fiber over $\lambda$. \end{defin}

It is easy to see that the identity function $f(z) = z$ is constant on each fiber. It follows that each continuous function is constant on each fiber as well. But the algebra $QC$ is strictly larger than $C$ and not all $QC$ functions are continuous on each fiber. For $QC$, we need to refine this partition.

\begin{defin} 
For each $\varphi \in M(H^\infty + C)$, define $E_\varphi = \{\psi \in M(H^\infty + C): \varphi(q) = \psi(q)~\mbox{for all}~q \in QC\}$. The set $E_\varphi$ is called the $QC$-level set corresponding to $\varphi$.
\end{defin}

Note that if $f \in QC$, then $\hat{f}$ is constant on a $QC$-level set. 

\begin{prop} A thin sequence can have at most one cluster point in a $QC$-level set. \end{prop}

\begin{proof} Suppose $\{\alpha_n\}$ is a thin sequence with two cluster points in $E_\varphi$. Then there are two distinct points, $\psi_1$ and $\psi_2$, in the closure of the sequence. But $M(H^\infty + C)$ is a Hausdorff space and therefore we can separate the two points by open sets $U_1$ and $U_2$ with disjoint closures and choose two disjoint subsets $\Lambda_1$ and $\Lambda_2$ of this sequence contained in $U_1$ and $U_2$, respectively. Now using Theorem~\ref{SW}, we obtain a function $f$ such that $f(\alpha_n) = 0$ if $\alpha_n \in \Lambda_1$ and $f(\alpha_n) = 1$ if $\alpha_n \in \Lambda_2$. In particular $\psi_1(f) = 0$ while $\psi_2(f) = 1$. But $f \in QC$ and therefore $f$ must be constant on the $QC$-level set. Since $\psi_1$ and $\psi_2$ belong to the same level set, this is impossible.
\end{proof}

The fact that the zeros of a thin Blaschke product that lie in $M(H^\infty + C)$ must lie in different $QC$-level sets is a very strong separation property. This paved the way for further study of the interpolation properties of thin sequences: Can we, as Shapiro and Shields did, transfer the study to the Hilbert space $H^2$? What about other $H^p$ spaces?

\section{Extending the definition of thin to $H^p$ spaces}

We have already hinted that thin sequences are the ones for which interpolation can be done with a very good bound on the norm. If we relax the interpolation condition a bit, we can study when functions do approximate interpolation with the best norm possible. To make this precise, we provide a definition that makes sense in a wider context -- for example, for general uniform algebras. (For more information, see \cite{GM}.)

\begin{defin} A sequence $\{\alpha_n\}$ is said to be an asymptotic interpolating sequence for $H^\infty$ if for every sequence $\{w_n\}$ in the ball of $\ell^\infty$, there is an $H^\infty$ function $f$ such that $\|f\|_\infty\le1$ and $|f(z_n) - w_n| \to 0$. \end{defin}

Following the work in \cite{GM}, Dyakonov and Nicolau showed that an interpolating sequence is thin if and only if
there is a sequence $\{m_j\}$, $0 < m_j < 1$ and $m_j \to 1$ such that every interpolation problem $F(z_j) = w_j$ with $|w_j| \le m_j$ has a solution $f \in H^\infty$ with $\|F\|_\infty \le 1$,  \cite{DN}. In fact, this happens if and only if there exists a sequence of positive numbers $\{\varepsilon_j\}$ such that every interpolation problem with $1 \ge |a_j| \ge \varepsilon_j$ for all $j$ has a nonvanishing solution $g \in H^\infty$. Thus, if the sequence $\{w_n\}$ grows slowly enough, we can do interpolation with the best norm possible. In fact, the solution can be chosen to be a thin Blaschke product, as noted in \cite{DN}. (For the details of the proof, see \cite{Mo}).

What are some other possible ways of defining thin sequences in the $H^p$ context?  We provide two possible alternative definitions below.

\begin{defin}
Let $1 \le p \le \infty$. A sequence $\{z_n\}$ is an eventual $1$-interpolating sequence  for $H^p$ $(EIS_p)$ if the following holds: For every $\varepsilon > 0$ there exists $N$ such that for each $\{a_n\} \in \ell^p$ there exists $f_{N, a} \in H^p$ with
$$f_{N, a}(z_n) (1 - |z_n|^2)^{1/p} = a_n ~\mbox{for}~ n \ge N ~\mbox{and}~ \|f_{N, a}\|_p \le (1 + \varepsilon) \|a_n\|_{N, \ell^p}.$$
\end{defin} 

\begin{defin} Let $1 \le p \le \infty$. A sequence $\{z_j\}$ is a strong $AIS_p$-sequence if for all $\varepsilon > 0$ there exists $N$ such that for all sequences $\{a_j\} \in \ell^p$ there exists a function $G_{N, a} \in H^p$ such that $\|G_{N, a}\|_p \le \|a\|_{N,\ell^p}$ and 
$$\|G_{N, a}(z_j) (1 - |z_j|^2)^{1/p} - a_j\|_{N, \ell^p} < \varepsilon \|a_j\|_{N, \ell^p}.$$ \end{defin}

It turns out that both of these ``new'' definitions are equivalent to a sequence being thin, see \cite{GPW}.

\section{Maximal ideal space and operator theory}

For $h \in L^\infty$ define the Toeplitz operator on $H^2$ by 
$T_h f = Ph f$, where $P$ is the orthogonal projection from  $L^2~\mbox{to}~ H^2.$
The Hankel operator is $H_h f = (I - P) h f,\,  f \in H^2.$ In 1963, Brown and Halmos \cite{BH} showed that if $f, g \in L^\infty$, then $T_f T_g = T_{fg}$ if and only if $\overline{f} \in H^\infty$ or $g \in H^\infty$. 
A natural question is the following: For which symbols $f, g$ is $T_f T_g$ a compact perturbation of a Toeplitz operator? In \cite{ACS}, Axler, Chang and Sarason showed that if $H^\infty[f] \cap H^\infty[g] \subset H^\infty+C,$ then $H_f^\star H_g$ is compact. Though they proved necessity for a large class of functions, the theorem was completed in 1982 by A. Volberg \cite{V}. These proofs relied on the maximal ideal space structure. There is a reason for this and it goes back to something we can see directly from the statement of the Chang-Marshall theorem.

\begin{cor}[Corollary to the Chang-Marshall Theorem] 
Let $A$ and $B$ be Douglas algebras. Then $M(A) \subseteq M(B)$ if and only if $B \subseteq A$.
\end{cor}

\begin{proof} Suppose $M(A) \subseteq M(B)$. Let $b$ be an interpolating Blaschke product invertible in $B$. Then $b$ cannot be in a maximal ideal of $B$. Therefore, since maximal ideals are precisely the kernels of the nonzero multiplicative linear functionals on $B$, we see that $b$ cannot vanish at any point of $M(B)$ -- and therefore the same is true on $M(A)$. Now $b \in H^\infty \subseteq A$ and since $b$ does not vanish on $M(A)$, $b$ is invertible in $A$. Thus $\overline{b} \in A$. Now we use the Chang-Marshall theorem to conclude that since $B$ is generated by $H^\infty$ and the conjugates of the interpolating Blaschke products invertible in $B$ -- all of which are invertible in $A$ as well, we have $B \subseteq A$.

For the other direction, suppose $B \subseteq A$. Let $\varphi \in M(A)$. Then for every Blaschke product $b$ invertible in $B$, we see that $b$ is also invertible in $A$. Therefore, $1 = \varphi(b \overline{b}) = \varphi(b) \varphi(\overline{b}) = |\varphi(b)|^2$. Thus, $|\varphi(b)| = 1$ and since $\varphi(b)$ is given by integration against a positive measure $\mu$ supported on a subset of the maximal ideal space, we see that $b$ must be constant on the support of $\varphi$. Thus, if $f, g \in B$, we know that $f$ and $g$ are limits of functions of the form $\sum_j h_j \overline{b_j}$ with $b_j$ Blaschke products invertible in $B$. By our argument above, the conjugates of the Blaschke products are all constant on the support of $\varphi$, and therefore -- as far as $\varphi$ is concerned -- they act like $H^\infty$ functions; that is,
$$\varphi(f g) = \int_{\mbox{supp} \, \varphi } fg d\mu_\varphi = \varphi(f) \varphi(g).$$ Thus, $\varphi$ is (or can be identified with) a nonzero multiplicative linear functional on $B$.\end{proof}

So let us return to what Axler, Chang, and Sarason and, later, Volberg wanted to do. They each wanted to show something about the algebra $H^\infty[f] \cap H^\infty[g]$. Since $H^\infty[f]$ and $H^\infty[g]$ are each Douglas algebras and the intersection is again a Douglas algebra, we expect the Chang-Marshall theorem to come into play here; that is, we expect a proof that relies on the techniques that were developing at the time. And that is precisely what happened -- their results depended on a distribution function inequality as well as maximal ideal space techniques and Volberg's proof used some of these same techniques.

\section{Asymptotically orthonormal sequences}

Volberg's paper not only answered the question of whether the converse of the Axler, Chang, Sarason result was valid, it also looked at so-called {\it asymptotically orthonormal sequences} and their connection to thin sequences and properties of the associated Gram matrix. We first recall some definitions.

\bigskip

Let $\{x_n\}$ be a sequence in a complex Hilbert space $\mathcal{H}$. 

\begin{defin}
The sequence $\{x_n\}$ is said to be a Riesz sequence if there are positive constants $c$ and $C$ for which 
$$c \sum_{n \ge 1} |a_n|^2 \le \left\|\sum_{n \ge 1} a_n x_n\right\|^2_{\mathcal{H}} \le C \sum_{n \ge 1} |a_n|^2$$ for all sequences $\{a_n\} \in \ell^2.$
\end{defin}

We are interested in the following setting: Let $K_z(w) = \frac{1}{1 - \overline{z} w}$ denote the reproducing kernel for $H^2$ for $z \in \mathbb{D}$, $k_z$ the normalized reproducing kernel, and given a sequence of points $\{z_j\}$, recall that $G$ denotes the Gram matrix with entries $k_{ij} = \langle k_{z_i}, k_{z_j} \rangle$. Riesz sequences correspond to the ones for which the associated Gram matrix is invertible. We are now ready to introduce our asymptotically orthonormal sequences.

\begin{defin}
A sequence $\{x_n\}$ is an asymptotically orthonormal sequence (AOS) in a Hilbert space $\mathcal{H}$ if there exists an integer $N_0$ such that for all $N \ge N_0$ there are constants $c_N$ and $C_N$ such that
$$c_N \sum_{n \ge N} |a_n|^2 \le \left\|\sum_{n \ge N} a_n x_n \right\|^2_{\mathcal{H}} \le C_N \sum_{n \ge N} |a_n|^2,$$ where $\displaystyle\lim_{N \to \infty} c_N = \lim_{N \to \infty} C_N = 1$.
\end{defin}
If we can take $N_0 = 1$, the sequence is an asymptotically orthonormal basic sequence, or AOB.

Volberg showed (see also \cite{CFT}) that the following is true.

\begin{thm}[Volberg, Theorem 2 in \cite{V}] 
\label{Volberg} 
The following are equivalent:
\begin{enumerate}
\item$\{z_n\}$ is a thin interpolating sequence;
\item
The sequence $\{k_{z_n}\}$ is a complete $AOB$ for its span.
\item
There exist a separable Hilbert space $\mathcal{K}$, an orthonormal basis $\{e_n\}$ for $\mathcal{K}$ and $U, K: \mathcal{K} \to K_B$, $U$ unitary, $K$ compact, $U + K$ invertible, such that 
$$(U + K)(e_n) = k_{z_n} \text{ for all } n \in \N.$$
\item The Gram matrix associated to the sequence defines a bounded invertible operator of the form $I + K$ with $K$ compact.
\end{enumerate}
 \end{thm}

The proof used  the main lemma from \cite{ACS} as well as  Hoffman's theory. Volberg also showed that $G - I \in S_2$ where $S_2$ denotes the Hilbert-Schmidt operators if and only if $\prod_j \delta_j$ converges. Thus, $G - I$ is in the Schatten class $S_2$ if and only if $\sum_j (1 - \delta_j) < \infty$. What about $2 < p < \infty$? 

Using Earl's theorem and results that are essentially in Shapiro and Shields (see also \cite{AM})  J. E. McCarthy, S. Pott, and the authors \cite{GMPW} showed the following:

\begin{thm}
Let $2 \le p < \infty$. Then $G - I \in S_p$ if and only if $\sum_n(1 - \delta_n^2)^{p/2} < \infty$.
\end{thm}

This theorem extends Volberg's result to the cases between $2$ and infinity and simplifies the proof for the case $p = \infty$.

\section{Carleson measures and thin sequences}

It is possible to characterize thin sequences in terms of a certain vanishing Carleson measure condition.  This Carleson measure condition has strong connections to the notions of eventual interpolating sequences and the property of strong $AIS_p$.


 For $z \in \mathbb{D}$, we let $I_z$ denote the interval in $\T$ with center $\frac{z}{|z|}$ and length $1 - |z|$. For an interval $I$ 
in $\T$, we let $$S_I = \left\{z \in \mathbb{D}: \frac{z}{|z|} \in I~\mbox{and}~ |z| \ge 1 - |I|\right\}.$$ For $A > 0$, the interval $AI$ denotes an interval with the same center as $I$ and length $A|I|$.
Given a positive measure $\mu$ on $\D$, let us denote the (possibly infinite) constant
$$
    \CC(\mu) =  \sup_{f \in H^2, f \neq 0} \frac{\|f\|^2_{L^2(\D, \mu)}}{\|f\|^2_2}
$$
as the Carleson embedding constant of $\mu$ on $H^2$ and
$$
    \RR(\mu) =  \sup_{z\in\mathbb{D}} \frac{\|K_z\|_{L^2(\D, \mu)}}{\|K_z\|_2}
$$
as the embedding constant of $\mu$ on the reproducing kernel of $H^2$.  We use $K_z$ for the non-normalized kernel later, and $k_z$ for the normalized kernel.

The Carleson Embedding Theorem asserts that the constants are equivalent. In particular, there exists a constant $c$ such that
$$
   \RR(\mu) \le \CC(\mu) \le c \RR(\mu),
$$
with best known constant $c= 2e$, \cite{PTW}.

We recall the following result  from \cite{SW}; for a generalized version, see \cite{CFT}.   This result provides a direction connection between thin sequences and a certain measure being a vanishing Carleson measure.
\begin{thm}[See Sundberg, Wolff, Lemma 7.1 in \cite{SW} or Chalendar, Fricain, Timotin, Proposition 4.2 in \cite{CFT}]
\label{Cmeasure} Suppose $Z = \{z_n\}$ is a sequence of distinct points. Then the following are equivalent:

\begin{enumerate} 
\item $Z$ is a thin interpolating sequence;
\item for any $A \ge 1$, $$\lim_{n \to \infty} \frac{1}{|I_{z_n}|} \sum_{k \ne n, z_k \in S(A I_n)} (1 - |z_k|) = 0.$$
\end{enumerate}
\end{thm}

Using this result it is possible to prove the following.  

\begin{thm}[\cite{GPW}]
\label{thm:Carleson} Suppose $Z = \{z_n\}$ is a sequence. For $N > 0$, let  $$\mu_N = \sum_{k \ge N} (1 - |z_k|^2)\delta_{z_k}.$$        
Then the following are equivalent:
\begin{enumerate} 
\item $Z$ is a thin sequence;
\item $ \CC(\mu_N) \to 1$ as $N \to \infty$;
\item $ \RR(\mu_N) \to 1$ as $N \to \infty$.
\end{enumerate}
\end{thm}
The proof of $(1)\Rightarrow (2)$ uses Volberg's characterization of thin sequences as those that are asymptotic orthonormal bases \cite{V}, while $(3)\Rightarrow (1)$ is a computation with the Weierstrass inequality.  And, of course $(2)\Rightarrow (3)$ is immediate.

With this characterization of thin sequences it is possible to provide the following list of equivalent conditions for a sequence to be thin. 

\begin{thm}[\cite{GPW}]
\label{main} Let $\{z_n\}$ be a Blaschke sequence of distinct points in $\mathbb{D}$. The following are equivalent:
\begin{enumerate} 
\item $\{z_n\}$ is an $EIS_p$ sequence for some $p$ with $1 \le p \le \infty$;
\item $\{z_n\}$ is thin;
\item $\{k_{z_n}\}$ is a complete AOB in $K_B$;
\item $\{z_n\}$ is a strong-$AIS_p$ sequence for some $p$ with $1 \le p \le \infty$;
\item The measure $$\mu_N = \sum_{k \ge N} (1 - |z_k|^2)\delta_{z_k}$$ is a Carleson measure with 
Carleson embedding constant $\CC(\mu_N)$ satisfying $\CC(\mu_N) \to 1$ as $N \to \infty$;
\item The measure $$\nu_N = \sum_{k \ge N}\frac{(1 - |z_k|^2)}{\delta_k} \delta_{z_k}$$ is a Carleson measure with embedding constant $\RR_{\nu_N}$ 
on reproducing kernels satisfying $\RR_{\nu_N} \to 1$.
\end{enumerate}
Moreover, if $\{z_n\}$ is an $EIS_p$ $($strong-$AIS_p$$)$ sequence for some $p$ with $1 \le p \le \infty$, then it is an $EIS_p$ $($strong $AIS_p$$)$ sequence for all $p$.
\end{thm}

\section{Future Directions: Model Spaces}

We conclude with a discussion of thin sequences in other contexts. Given a (nonconstant) inner function $\Theta$, one can also study thin sequences in model spaces, where the model space for $\Theta$ an inner function is defined by $K_\Theta = H^2 \ominus \Theta H^2$. The reproducing kernel in $K_\Theta$ for $\lambda_0 \in \mathbb{D}$ is 
$$
K_{\lambda_0}^\Theta(z) = \frac{1 - \overline{\Theta(\lambda_0)}{\Theta(z)}}{1 - \overline{\lambda_0}z}
$$ 
and the normalized reproducing kernel is 
$$
k_{\lambda_0}^\Theta(z) = \sqrt{\frac{1 - |\lambda_0|^2}{1 - |\Theta(\lambda_0)|^2}} K_{\lambda_0}^\Theta(z).
$$
Finally, note that $$K_{\lambda_0} = K_{\lambda_0}^\Theta + \Theta \overline{\Theta(\lambda_0)}K_{\lambda_0}.$$
We let $P_\Theta$ denote the orthogonal projection of $H^2$ onto $K_\Theta$.

Asymptotically orthonormal sequences were studied in \cite{F} and \cite{CFT}. We mention here one theorem that encompasses many of these results. Proofs or references for proofs can be found in \cite{GW}. We remark that we get Theorem 4.6 of \cite{GPW} when we let $\Theta = B$ in the proof below (which is simply Theorem \ref{main} above).

\begin{thm}[Theorem 3.5 in \cite{GW}]\label{GW} Let $\{\lambda_n\}$ be an interpolating sequence  in $\mathbb{D}$ and let $\Theta$ be an inner function.  Suppose that $\kappa:=\sup_{n} \left\vert \Theta(\lambda_n)\right| < 1$.  The following are equivalent:
\begin{enumerate} 

\item $\{\lambda_n\}$ is an $EIS_{H^2}$ sequence\label{eish2};
\item $\{\lambda_n\}$ is a thin interpolating sequence\label{thin1};
\item \label{aob} Either
\begin{enumerate}
\item $\{k_{\lambda_n}^\Theta\}_{n\ge1}$ is an $AOB$, or
\item there exists $p \ge 2$ such that $\{k_{\lambda_n}^\Theta\}_{n \ge p}$ is a complete $AOB$ in $K_\Theta$;
\end{enumerate}
\item $\{\lambda_n\}$ is an $AIS_{H^2}$ sequence\label{Aish2};
\item The measure $$\mu_N = \sum_{k \ge N} (1 - |\lambda_k|^2)\delta_{\lambda_k}$$ is a Carleson measure for $H^2$ with 
Carleson embedding constant $\CC(\mu_N)$ satisfying $\CC(\mu_N) \to 1$ as $N \to \infty$\label{C1};
\item The measure $$\nu_N = \sum_{k \ge N}\frac{(1 - |\lambda_k|^2)}{\delta_k} \delta_{\lambda_k}$$ is a Carleson measure for $H^2$ with embedding constant $\RR_{\nu_N}$ 
on reproducing kernels satisfying $\RR_{\nu_N} \to 1$\label{C2}.

\bigskip
\noindent
Further, \eqref{eis} and \eqref{ais} are equivalent to each other and imply each of the statements above. If, in addition, $ \Theta(\lambda_n) \to0$, then  \eqref{eish2} - \eqref{ais} are equivalent.
\bigskip

\item $\{\lambda_n\}$ is an $EIS_{K_\Theta}$ sequence\label{eis};
\item $\{\lambda_n\}$ is an $AIS_{K_\Theta}$ sequence\label{ais}.
\end{enumerate}
\end{thm}

There are many directions for future research. For example, connections to truncated Toeplitz operators have been studied by Lopatto and Rochberg \cite{LR} as well as R. Bessonov \cite{B}. In addition, we mention two questions below.

\begin{question} One can define thin sequences in other spaces (for example, Bergman spaces) and see whether the results that we have discussed here extend to those spaces: If a sequence is a thin sequence in a space $X$, is there a particularly good bound on the interpolation constant? \end{question}

\begin{question} Finally, we note that thin sequences are those satisfying $\delta_j \to 1$ and they are interpolating sequences for an important space of functions, $QA$. If $\sum_j (1 - \delta_j)^p < \infty$, is the sequence interpolating for some natural function space? \end{question}

\end{document}